\documentclass[11pt]{amsart}
\usepackage{amssymb,amsthm}
\usepackage[all]{xy}
\usepackage{amsfonts}
\theoremstyle{plain}
\newtheorem{Thm}{Theorem}[section]
\newtheorem{Lem}[Thm]{Lemma}

\newtheorem{Prop}[Thm]{Proposition}

\theoremstyle{definition}
\newtheorem{Def}[Thm]{Definition}

\title[Set-valued maps]{Set-valued maps and some generalized metric spaces}
\author[E.-G. Yang]{Er-Guang YANG}
\address[E.-G. Yang]{ School of Mathematics $\&$ Physics, Anhui University of Technology, Maanshan 243002, P.R. China}
\email{egyang@126.com}
\thanks{}
\keywords{Set-valued maps; Strictly increasing closed cover; $\gamma$-spaces; Stratifiable spaces; Semi-metric spaces; Metric spaces.}
\subjclass[2000]{54C08, 54C60, 54E20, 54E25, 54E35, 54E99}
\date{}

\begin{document}

\maketitle

\begin{abstract}
To give characterizations of monotonically countably paracompact spaces with set-valued maps, Yamazaki \cite{Y} introduced the notion of strictly
increasing closed cover of a topological space with which the boundedness of a set-valued map was defined. In this paper, we show that most of
generalized metric spaces such as stratifiable spaces, semi-metrizable spaces can be characterized with set-valued maps with values into the family
of all closed nonempty subsets of a space which has a strictly increasing closed cover. Moreover, as an application, we use the results obtained to
give characterizations of the corresponding spaces with generalized real-valued functions.
\end{abstract}

\section{Introduction and preliminaries}
Throughout, a space always means a Hausdorff topological space. For a space $X$, we denote by ${\mathcal C}_X$ the family of all compact subsets of
$X$ and $\mathcal S_X=\{\{x\}:x\in X\}$. $\tau$ and $\tau^c$ denote the topology of $X$ and the family of all closed subsets of $X$ respectively.
For a subset $A$ of a space $X$, we write $\overline A$ ($int(A)$) for the closure (interior) of $A$ in $X$. The set of all positive integers is
denoted by $\mathbb{N}$. For a space $Y$, we use $2^Y$ ($\mathcal F(Y)$) to denote the family of all nonempty (nonempty closed) subsets of $Y$.

For a set-valued map $\phi:X\rightarrow 2^Y$ and $B\subset Y$, denote $\phi^{-1}[B]=\{x\in X:\phi(x)\cap B\neq\emptyset\}$ and $\phi^{\#}[B]=\{x\in X
:\phi(x)\subset B\}$. $\phi$ is called lower (upper) semi-continuous if $\phi^{-1}[V]$ ($\phi^{\#}[V]$) is open in $X$ for each open subset $V$ of $Y$.

For maps $\phi,\varphi:X\rightarrow 2^Y$, we express $\phi\subset\varphi$, if $\phi(x)\subset\varphi(x)$ for each $x\in X$.

Set-valued maps are known to be useful in the characterizations of some topological spaces. Many covering properties such as paracompactness \cite{Mi},
strong paracompactness \cite{Gu} and collectionwise normality \cite{Miy} can be characterized with set-valued maps.

To give characterizations of monotonically countably paracompact spaces with set-valued maps, Yamazaki \cite{Y} introduced the notion of strictly
increasing closed cover of a topological space with which the boundedness of a set-valued map was defined.

A sequence $\{B_n\}_{n\in\mathbb{N}}$ of closed subsets of a space $Y$ is called a strictly increasing closed cover of $Y$ if $Y=\bigcup_{n\in\mathbb
{N}}B_n$, $B_1\neq\emptyset$ and $B_n\subsetneq B_{n+1}$ for each $n\in\mathbb{N}$.

Let $X$ be a space and $Y$ a space with a strictly increasing closed cover $\{B_n\}_{n\in\mathbb{N}}$. A set-valued map $\phi:X\rightarrow\mathcal F(
Y)$ is called bounded on a subset $A$ of $X$ if $A\subset\phi^{\#}[B_n]$ for some $n\in\mathbb{N}$. $\phi$ is called locally bounded at a point $x\in
X$ if there exists a neighborhood $O$ of $x$ such that $\phi$ is bounded on $O$. $\phi$ is called locally bounded on a subset $A$ of $X$ if $\phi$ is
locally bounded at each $x\in A$. Denote $\mathcal B(Y;\{B_n\})=\{F\in\mathcal F(Y):F\subset B_n \ {\rm for \ some} \ n\in\mathbb{N}\}$.

By expansions of locally bounded maps $\phi:X\rightarrow\mathcal B(Y;\{B_n\})$, Yamazaki \cite{Y} presented some characterizations of (monotonically)
countably paracompact spaces and (monotonically) countably metacompact spaces. It was also asked that whether stratifiable spaces can be
characterized with set-valued maps along the same lines. In \cite{X}, Xie and Yan gave a positive answer to this question by presenting some
characterizations of stratifiable spaces with set-valued maps.

In this paper, we shall show that besides the spaces mentioned above, many other generalized metric spaces can also be characterized with
set-valued maps with values in $\mathcal F(Y)$, where $Y$ is a space with a strictly increasing closed cover.

Let $X$ be a space and $Y$ a space having a strictly increasing closed cover $\{B_n\}_{n\in\mathbb{N}}$. We write $\mathcal L(X,\mathcal F(Y))$ for
the set of all lower semi-continuous maps from $X$ to $\mathcal F(Y)$. Let $\mathcal A$ be a family of subsets of $X$ and $\mathcal M$ a family of
maps from $X$ to $\mathcal F(Y)$. For a map $\phi:\mathcal A\rightarrow\mathcal M$ and $A\in\mathcal A$, we write $\phi_A$ instead of $\phi(A)$. If
$\{x\}\in\mathcal A$, then we write $\phi_x$ for $\phi_{\{x\}}$. Let $\mathcal A,\mathcal B$ be two families of subsets of $X$ and suppose that
$\mathcal S_X\subset\mathcal A$. Consider the following conditions for a map $\phi:\mathcal A\rightarrow\mathcal M$.

$(G_{\mathcal A})$ For each $A\in\mathcal A$, $A=\bigcap_{n\in\mathbb{N}}\phi_A^{-1}[Y\setminus B_n]$.

$(m_{\mathcal A})$ If $A_1,A_2\in\mathcal A$ and $A_1\subset A_2$, then $\phi_{A_1}\subset\phi_{A_2}$.

$(b_{\mathcal A}(\mathcal B))$ For each $A\in\mathcal A$ and $B\in\mathcal B$, if $A\cap B=\emptyset$, then $\phi_A$ is bounded on $B$.

$(b'_{\mathcal A}(\mathcal B))$ For each $A\in\mathcal A$ and $B\in\mathcal B$, if $A\cap B=\emptyset$, then there exists an open set $V\supset B$
such that $\phi_A$ is bounded on $V$.

$(S)$ For each $x,y\in X$, $\phi_x(y)=\phi_y(x)$.

$(E)$ For each $x,y,z\in X$, $\phi_x(y)\subset\phi_x(z)$ or $\phi_y(z)\subset\phi_x(z)$.

Notice that if a map $\phi:\mathcal A\rightarrow\mathcal M$ satisfies $(G_{\mathcal A})$ (resp., $(m_{\mathcal A})$, $(b_{\mathcal A}(\mathcal B))$,
$(b'_{\mathcal A}(\mathcal B))$) and $\mathcal C\subset\mathcal A$, then the restriction $\phi|_{\mathcal C}:\mathcal C\rightarrow\mathcal M$
satisfies $(G_{\mathcal C})$ (resp., $(m_{\mathcal C})$, $(b_{\mathcal C}(\mathcal B))$, $(b'_{\mathcal C}(\mathcal B))$). In the sequel, for a map
$\phi:\mathcal A\rightarrow\mathcal M$ satisfying $(G_{\mathcal A})$ (and the others) and $\mathcal C\subset\mathcal A$, when we say that $\phi$
satisfies $(G_{\mathcal C})$ (and the others), it will always mean that $\phi|_{\mathcal C}$ satisfies the corresponding conditions.

A $g$-function for a space $X$ is a map $g:\mathbb{N}\times X\rightarrow\tau$ such that for every $x\in X$ and $n\in\mathbb{N}$, $x\in g(n,x)$
and $g(n+1,x)\subset g(n,x)$. For $A\subset X$, $g(n,A)=\cup\{g(n,x): x\in A\}$.

{\bf Convention}: In the following, a space $Y$ always has a strictly increasing closed cover $\{B_n\}_{n\in\mathbb{N}}$.

The following lemma will be frequently used in the sequel.

\begin{Lem}\label{L1}
Let $\mathcal A$ be a family of subsets of a space $X$. Suppose that for each $A\in\mathcal A$, there exists a decreasing sequence $\{G(n,A):n
\in\mathbb{N}\}$ of open subsets of $X$ such that $A=\bigcap_{n\in\mathbb{N}}G(n,A)$. For each $A\in\mathcal A(X)$ and $x\notin A$, let $n_A(x)=\min
\{n\in\mathbb{N}:x\notin G(n,A)\}$. For each $A\in\mathcal A(X)$, define a map $\phi_A:X\rightarrow\mathcal F(Y)$ by letting $\phi_A(x)=Y$ whenever
$x\in A$ and $\phi_A(x)=B_{n_A(x)}$ whenever $x\notin A$. Then $\phi:\mathcal A\rightarrow\mathcal L(X,\mathcal F(Y))$ satisfies $(G_{\mathcal A})$.
If in addition, $G(n,A_1)\subset G(n,A_2)$ for each $n\in\mathbb{N}$ whenever $A_1\subset A_2$, then $\phi$ satisfies $(m_{\mathcal A})$.
\end{Lem}
\begin{proof}

Let $A\in\mathcal A$. To show that $\phi_A$ is lower semi-continuous, let $W$ be an open subset of $Y$, $x\in X$ and suppose that $\phi_A(x)\cap
W\neq\emptyset$.

Case 1. $x\in A$. Then $\phi_A(x)=Y$. Choose $m\in\mathbb{N}$ such that $B_m\cap W\neq\emptyset$. Then $O=G(m,A)$ is an open neighborhood of $x$.
For each $y\in O$, if $y\in A$, then $\phi_A(y)=Y$. If $y\notin A$, then $m<n_A(y)$ and hence $\phi_A(y)=B_{n_A(y)}\supset B_m$. It follows that
$\phi_A(y)\cap W\neq\emptyset$.

Case 2. $x\notin A$. Then $\phi_A(x)=B_{n_A(x)}$. If $n_A(x)=1$, then for each $y\in X$, $\phi_A(y)=Y$ or $\phi_A(y)=B_{n_A(y)}\supset\phi_A(x)$.
If $n_A(x)>1$, then $O=G(n_A(x)-1,A)$ is an open neighborhood of $x$. For each $y\in O$, if $y\in A$, then $\phi_A(y)=Y$. If $y\notin A$, then
$n_A(x)-1<n_A(y)$. Hence $\phi_A(y)=B_{n_A(y)}\supset B_{n_A(x)}=\phi_A(x)$. It follows that $\phi_A(y)\cap W\neq\emptyset$.

By Case 1 and Case 2, we conclude that $\phi_A$ is lower semi-continuous.

By the definition of $\phi_A$, if $x\in A$ then $\phi_A(x)=Y$ from which it follows that $\phi_A(x)\setminus B_n\neq\emptyset$ and thus $x\in\phi_A^
{-1}[Y\setminus B_n]$ for each $n\in\mathbb{N}$. If $x\notin A$ then $\phi_A(x)=B_{n_A(x)}$ from which it follows that $x\notin\phi_A^{-1}[Y\setminus
B_{n_A(x)}]$. This implies that $A=\bigcap_{n\in\mathbb{N}}\phi_A^{-1}[Y\setminus B_n]$.

Now, assume in addition that $A_1\subset A_2$ implies that $G(n,A_1)\subset G(n,A_2)$ for each $n\in \mathbb{N}$. Let $A_1,A_2\in\mathcal A$ with $A_1
\subset A_2$. For each $x\in X$, if $x\in A_2$, then $\phi_{A_2}(x)=Y\supset\phi_{A_1}(x)$. If $x\notin A_2$, then $x\notin A_1$. Since $x\notin G(n_{A_2}
(x),A_2)\supset G(n_{A_2}(x),A_1)$, we have that $n_{A_1}(x)\leq n_{A_2}(x)$ from which it follows that $\phi_{A_1}(x)=B_{n_{A_1}(x)}\subset B_{n_{A_2}(x)
}=\phi_{A_2}(x)$. Therefore, $\phi_{A_1}\subset\phi_{A_2}$. This implies that $\phi$ satisfies $(m_{\mathcal A})$.
\end{proof}

\section{The class of first countable spaces}
In this section, we prove some results for first countable spaces and some spaces contained in the the class of first countable spaces.

\begin{Prop}\label{P1}
A space $X$ is regular first countable if and only if there exists a map $\phi:\mathcal S_X\rightarrow\mathcal L(X,\mathcal F
(Y))$ satisfying $(G_{\mathcal S_X})$ and $(b'_{\mathcal S_X}(\tau^c))$.
\end{Prop}
\begin{proof}
For each $x\in X$, let $\{V_n(x):n\in\mathbb{N}\}$ be a decreasing open neighborhood base of $x$. For each $x,y\in X$ with $x\neq y$, let $n_x(y)=\min\{n\in
\mathbb{N}:y\notin V_n(x)\}$. For each $x\in X$, define a map $\phi_x:X\rightarrow\mathcal F(Y)$ by letting $\phi_x(x)=Y$ and $\phi_x(y)=B_{n_x(y)}$ whenever
$y\neq x$. By Lemma \ref{L1}, $\phi:\mathcal S_X\rightarrow\mathcal L(X,\mathcal F(Y))$ satisfies $(G_{\mathcal S_X})$.

Suppose that $x\notin F\in\tau^c$. Then there exists an open set $V\supset F$ and $m\in\mathbb{N}$ such that $V\cap V_m(x)=\emptyset$. For each $y
\in V$, $y\notin V_m(x)$ from which it follows that $n_x(y)\leq m$. Hence, $\phi_x(y)=B_{n_x(y)}\subset B_m$. This implies that $\phi_x$ is bounded on $V$.

Conversely, for each $x\in X$ and $n\in\mathbb{N}$, let $V_n(x)=\phi_x^{-1}[Y\setminus B_n]$. Since $\phi_x$ is lower semi-continuous, we
have that $V_n(x)$ is open. By $(G_{\mathcal S_X})$, $x\in\phi_x^{-1}[Y\setminus B_n]=V_n(x)$ for each $n\in\mathbb{N}$.

Suppose that $x\notin F\in\tau^c$. By $(b'_{\mathcal S_X}(\tau^c))$, there exists an open set $V\supset F$ and $m\in\mathbb{N}$ such that $V\subset\phi_x
^{\#}[B_m]$. It follows that $V\cap V_m(x)=\emptyset$. Therefore, $X$ is a regular first countable space.
\end{proof}

An analogous argument proves the following.

\begin{Prop}\label{P2}
A space $X$ is first countable if and only if there exists a map $\phi:\mathcal S_X\rightarrow\mathcal L(X,\mathcal F(Y))$ satisfying $(G_{
\mathcal S_X})$ and $(b_{\mathcal S_X}(\tau^c))$.
\end{Prop}

A space $X$ is called a $\gamma$-space \cite{Ho} if there exists a $g$-function $g$ for $X$ such that if $y_n\in g(n,x)$ and $x_n\in g(n,y_n)$ for all
$n\in\mathbb{N}$, then $x$ is a cluster point of $\langle x_n\rangle$.

\begin{Lem}\label{L2}\cite{Li}
A space $X$ is a $\gamma$-space if and only if there exists a $g$-function $g$ for $X$ such that for each $K\in{\mathcal C}_X$ and $F\in\tau^c$ with $K\cap
F=\emptyset$, there exists $m\in\mathbb{N}$ such that $F\cap g(m,K)=\emptyset$.
\end{Lem}

\begin{Thm}\label{T1}
A space $X$ is a regular $\gamma$-space if and only if there exists a map $\phi:\mathcal C_X\rightarrow\mathcal L(X,\mathcal F(Y))$ satisfying $(G_
{\mathcal C_X})$, $(m_{\mathcal C_X})$ and $(b'_{\mathcal C_X}(\tau^c))$.
\end{Thm}
\begin{proof}
Let $g$ be the $g$-function in Lemma \ref{L2}. It is easy to see that for each $K\in{\mathcal C}_X$, $K=\bigcap_{n\in\mathbb{N}}g(n,K)$. For each $K\in{
\mathcal C}_X$ and $x\notin K$, let $n_K(x)=\min\{n\in\mathbb{N}:x\notin g(n,K)\}$. For each $K\in{\mathcal C}_X$, define a map $\phi_K:X\rightarrow\mathcal
F(Y)$ by letting $\phi_K(x)=Y$ whenever $x\in K$ and $\phi_K(x)=B_{n_K(x)}$ whenever $x\notin K$. By Lemma \ref{L1}, $\phi:\mathcal C_X\rightarrow\mathcal L
(X,\mathcal F(Y))$ satisfies $(G_{\mathcal C_X})$ and $(m_{\mathcal C_X})$.

Suppose that $K\in{\mathcal C}_X$, $F\in\tau^c$ and $K\cap F=\emptyset$. Since $X$ is regular, there exist open subsets $U,V$ of $X$ such that $K\subset U$,
$F\subset V$ and $V\cap U=\emptyset$. By Lemma \ref{L2}, there is $m\in\mathbb{N}$ such that $g(m,K)\subset U$ and thus $V\cap g(m,K)=\emptyset$. For each $x
\in V$, $x\notin g(m,K)$ from which it follows that $n_K(x)\leq m$ and thus $\phi_K(x)=B_{n_K(x)}\subset B_m$. This implies that $\phi_K$ is bounded on $V$.

Conversely, for each $x\in X$ and $n\in\mathbb{N}$, let $g(n,x)=\phi_x^{-1}[Y\setminus B_n]$. Since $\phi_x$ is lower semi-continuous, $g(n,x)$ is open. By $(
G_{\mathcal C_X})$, $x\in g(n,x)$. It is clear that $g(n+1,x)\subset g(n,x)$. Thus $g$ is a $g$-function for $X$. Let $K\in{\mathcal C}_X$, $F\in\tau^c$ and $K
\cap F=\emptyset$. By $(b'_{\mathcal C_X}(\tau^c))$, there exists an open set $V\supset F$ and $m\in\mathbb{N}$ such that $V\subset\phi_K^{\#}[B_m]$. By $(m_{
\mathcal C_X})$, $V\subset\phi_x^{\#}[B_m]$ for each $x\in K$. It follows that $V\cap g(m,x)=\emptyset$ for each $x\in K$ and hence $V\cap g(m,K)=\emptyset$.
This implies that $X$ is regular. By Lemma \ref{L2}, $X$ is a $\gamma$-space.
\end{proof}

Similarly, we have the following.

\begin{Thm}\label{T2}
A space $X$ is a $\gamma$-space if and only if there exists a map $\phi:\mathcal C_X\rightarrow\mathcal L(X,\mathcal F(Y))$ satisfying $(G_{\mathcal
C_X})$, $(m_{\mathcal C_X})$ and $(b_{\mathcal C_X}(\tau^c))$.
\end{Thm}

A space $X$ is called strongly first countable \cite{Ho} if there exists a $g$-function $g$ for $X$ such that for each $x\in X$, $\{g(n,x):n\in\mathbb{N}\}$
is a neighborhood base of $x$ and if $y\in g(n,x)$, then $g(n,y)\subset g(n,x)$.

\begin{Thm}\label{T3}
A space $X$ is strongly first countable if and only if there exists a map $\phi:\mathcal S_X\rightarrow\mathcal L(X,\mathcal F(Y))$ satisfying $(
G_{\mathcal S_X})$, $(E)$ and $(b_{\mathcal S_X}(\tau^c))$.
\end{Thm}
\begin{proof} Let $g$ be the $g$-function for a strongly first countable space. For each $x,y\in X$ with $x\neq y$, let $n_x(y)=\min\{n\in\mathbb{N}:y\notin
g(n,x)\}$. For each $x\in X$, define a map $\phi_x:X\rightarrow\mathcal F(Y)$ by letting $\phi_x(x)=Y$ and $\phi_x(y)=B_{n_x(y)}$ whenever $y\neq x$. Then $
\phi:\mathcal S_X\rightarrow\mathcal L(X,\mathcal F(Y))$ satisfies $(G_{\mathcal S_X})$ and $(b_{\mathcal S_X}(\tau^c))$.

Let $x,y,z\in X$. We may assume that $x,y,z$ are pairwise distinct. Assume that $\phi_x(y)\nsubseteq\phi_x(z)$ and $\phi_y(z)\nsubseteq\phi_x(z)$. Then $n_x
(z)<n_x(y)$ and $n_x(z)<n_y(z)$. Thus $y\in g(n_x(z),x)$ and $z\in g(n_x(z),y)$ from which it follows that $z\in g(n_x(z),x)$, a contradiction.

Conversely, define a $g$-function for $X$ by letting $g(n,x)=\phi_x^{-1}[Y\setminus B_n]$ for each $x\in X$ and $n\in\mathbb{N}$. From $(b_{\mathcal S_X}(
\tau^c))$ it follows that $\{g(n,x):n\in\mathbb{N}\}$ is a neighborhood base of $x$. Suppose that $y\in g(n,x)$ and $z\in g(n,y)$. Then $\phi_x(y)\setminus
B_n\neq\emptyset$ and $\phi_y(z)\setminus B_n\neq\emptyset$. By $(E)$, $\phi_x(z)\setminus B_n\neq\emptyset$ from which it follows that $z\in\phi_x^
{-1}[Y\setminus B_n]=g(n,x)$. This implies that $g(n,y)\subset g(n,x)$. Therefore, $X$ is strongly first countable.
\end{proof}

\section{The class of semi-stratifiable spaces}
This section is devoted to the characterizations of some spaces contained in the class of semi-stratifiable spaces.

\begin{Def}\label{d1}
A space $X$ is called stratifiable \cite{B} (semi-stratifiable \cite{C}) if there is a map $\rho:\mathbb{N}\times\tau^c\rightarrow
\tau$ such that

(1) $F=\bigcap_{n\in\mathbb{N}}\rho(n,F)=\bigcap_{n\in\mathbb{N}}\overline{\rho(n,F)}$ ($F=\bigcap_{n\in\mathbb{N}}\rho(n,F)$) for each
$F\in\tau^c$;

(2) if $F,H\in\tau^c$ and $F\subset H$, then $\rho(n,F)\subset\rho(n,H)$ for all $n\in \mathbb{N}$.

$X$ is called $k$-semi-stratifiable \cite{Lu} if there is a map $\rho:\mathbb{N}\times\tau^c\rightarrow\tau$ satisfies the conditions for a
semi-stratifiable space and

(3) for each $K\in{\mathcal C}_X$ and $F\in\tau^c$ with $K\cap F=\emptyset$, there is $m\in\mathbb{N}$ such that $K\cap\rho(m,F)=\emptyset$.
\end{Def}

The map $\rho$ in the above definition is called the stratification, semi-stratification, $k$-semi-stratification for $X$ respectively. Notice
that, without loss of generality, we may assume that $\rho$ is decreasing with respect to $n$.

\begin{Thm}\label{T4}
A space $X$ is $k$-semi-stratifiable if and only if there exists a map $\phi:\tau^c\rightarrow\mathcal L(X,\mathcal F(Y))$ satisfying $(G_{
\tau^c})$, $(m_{\tau^c})$ and $(b_{\tau^c}(\mathcal C_X))$.
\end{Thm}
\begin{proof}
Let $\rho$ be the $k$-semi-stratification for $X$ which is decreasing with respect to $n$. For each $F\in\tau^c$ and $x\notin F$, let $n_F(x)=\min\{n\in\mathbb
{N}:x\notin\rho(n,F)\}$. For each $F\in\tau^c$, define a map $\phi_F:X\rightarrow\mathcal F(Y)$ by letting $\phi_F(x)=Y$ whenever $x\in F$ and $\phi_F(x)=B
_{n_F(x)}$ whenever $x\notin F$. By Lemma \ref{L1}, $\phi:\tau^c\rightarrow\mathcal L(X,\mathcal F(Y))$ satisfies $(G_{\tau^c})$ and $(m_{\tau^c})$.

Let $F\in\tau^c$, $K\in{\mathcal C}_X$ and $F\cap K=\emptyset$. Then $K\cap\rho(m,F)=\emptyset$ for some $m\in\mathbb{N}$. For each $x\in K$, $x\notin\rho(m,F)$
from which it follows that $n_F(x)\leq m$. Hence, $\phi_F(x)=B_{n_F(x)}\subset B_m$. This implies that $\phi_F$ is bounded on $K$.

Conversely, for each $F\in\tau^c$ and $n\in\mathbb{N}$, let $\rho(n,F)=\phi_F^{-1}[Y\setminus B_n]$. Then $\rho(n,F)\in\tau$.

By $(m_{\tau^c})$, it is clear that if $F,H\in\tau^c$ and $F\subset H$, then $\rho(n,F)\subset\rho(n,H)$.

Let $F\in\tau^c$. By $(G_{\tau^c})$, $F=\bigcap_{n\in\mathbb{N}}\phi_F^{-1}[Y\setminus B_n]=\bigcap_{n\in\mathbb{N}}\rho(n,F)$. Let $F\in\tau^c$, $K\in{\mathcal
C}_X$ and $F\cap K=\emptyset$. By $(b_{\tau^c}(\mathcal C_X))$, $K\subset\phi_F^{\#}[B_m]$ for some $m\in\mathbb{N}$. It follows that $K\cap\rho(m,F)=\emptyset$.

By Definition \ref{d1}, $X$ is a $k$-semi-stratifiable space.
\end{proof}

As for semi-stratifiable spaces, we have the following.

\begin{Prop}\label{P3}
A space $X$ is semi-stratifiable if and only if there exists a map $\phi:\tau^c\rightarrow\mathcal L(X,\mathcal F(Y))$ satisfying $(G_{
\tau^c})$ and $(m_{\tau^c})$.
\end{Prop}

\begin{Thm}\label{T5}
A space $X$ is stratifiable if and only if there exists a map $\phi:\tau^c\rightarrow\mathcal L(X,\mathcal F(Y))$ satisfying $(G_{
\tau^c})$, $(m_{\tau^c})$ and $(\diamond)$ for each $F\in\tau^c$, $\phi_F$ is locally bounded on $X\setminus F$.
\end{Thm}
\begin{proof}
Let $\rho$ be the stratification for $X$ which is decreasing with respect to $n$. Define a map $\phi:\tau^c\rightarrow\mathcal L(X,\mathcal F(Y))$ as that in
the proof of the sufficiency of Theorem \ref{T4}. Then we only need to show $(\diamond)$.

Let $F\in\tau^c$. For each $x\in X\setminus F$, there is $m\in\mathbb{N}$ such that $x\notin\overline{\rho(m,F)}$. Then $O=X\setminus\overline{\rho(m,F)}$ is an
open neighborhood of $x$. For each $y\in O$, $n_F(y)\leq m$ from which it follows that $\phi_F(y)\subset B_m$. This implies that $\phi_F$ is locally bounded
at $x$.

Conversely, for each $F\in\tau^c$ and $n\in\mathbb{N}$, define an open subset $\rho(n,F)$ of $X$ as that in the proof of the sufficiency of Theorem \ref{T4}.
Then we only need to show that $\bigcap_{n\in\mathbb{N}}\overline{\rho(n,F)}\subset F$.

Let $F\in\tau^c$. By $(\diamond)$, if $x\notin F$, then $x\in int(\phi_F^{\#}[B_m])$ for some $m\in\mathbb{N}$. It follows that $x\notin X\setminus int(\phi_F^
{\#}[B_m])=\overline{\phi_F^{-1}[Y\setminus B_m]}=\overline{\rho(m,F)}$. This implies that $\bigcap_{n\in\mathbb{N}}\overline{\rho(n,F)}\subset F$.
\end{proof}

\section{Spaces in both classes}
In this section, we present characterizations of some spaces contained in both the the class of first countable spaces and the class of semi-stratifiable spaces
such as semi-metrizable spaces, Nagata-spaces and strongly quasi-metrizable spaces.

A space $X$ is called semi-metrizable \cite{W} if there is a function $d:X\times X\rightarrow [0,\infty)$ such that (1) $d(x,y)=0$ if and only if $x=y$; (2) $d(
x,y)=d(y,x)$ for all $x,y\in X$; (3) for each $x\in X$, $\{B(x,r):r>0\}$ is a neighborhood base of $x$, where $B(x,r)=\{y\in X: d(x,y)<r\}$. $X$ is called
$o$-semimetrizable \cite{G} if there is a semi-metric on $X$ such that for each $x\in X$ and $r>0$, $B(x,r)$ is open. $X$ is called $K$-semimetrizable \cite{M}
if there is a semi-metric $d$ on $X$ such that $d(K,H)>0$ for every disjoint pair $K,H$ of nonempty compact subsets of $X$.

\begin{Thm}\label{T6}
A space $X$ is semi-metrizable if and only if there exists a map $\phi:\tau^c\rightarrow\mathcal L(X,\mathcal F(Y))$ satisfying $(G_{
\tau^c})$, $(m_{\tau^c})$ and $(b_{\mathcal S_X}(\tau^c))$.
\end{Thm}
\begin{proof}
For each $F\in\tau^c$ and $n\in\mathbb{N}$, let $U(n,F)=\cup\{int(B(x,\frac{1}{n})):x\in F\}$. Then $\bigcap_{n\in\mathbb{N}}U(n,F)=F$. For each $x\notin F$, let
$n_F(x)=\min\{n\in\mathbb{N}:x\notin U(n,F)\}$. For each $F\in\tau^c$, define a map $\phi_F:X\rightarrow\mathcal F(Y)$ by letting $\phi_F(x)=Y$ whenever $x\in F$
and $\phi_F(x)=B_{n_F(x)}$ whenever $x\notin F$. By Lemma \ref{L1}, $\phi$ satisfies $(G_{\tau^c})$ and $(m_{\tau^c})$.

Let $x\notin F\in\tau^c$. Then there exists $m\in\mathbb{N}$ such that $B(x,\frac{1}{m})\cap F=\emptyset$. It follows that $y\notin int(B(x,\frac{1}{m}))=U(
m,\{x\})$ for each $y\in F$. Hence $n_{\{x\}}(y)\leq m$ and so $\phi_x(y)=B_{n_{\{x\}}(y)}\subset B_m$. This implies that $\phi_x$ is bounded on $F$.

Conversely, since $\phi$ satisfies $(G_{\mathcal S_X})$ and $(b_{\mathcal S_X}(\tau^c))$ (that is, $\phi|_{\mathcal S_X}$ satisfies $(G_{\mathcal S_X})$ and
$(b_{\mathcal S_X}(\tau^c))$), $X$ is first countable by Proposition \ref{P2}. By Proposition \ref{P3}, $X$ is semi-stratifiable. Thus $X$ is semi-metrizable
\cite{H}.
\end{proof}

\begin{Thm}\label{T7}
A space $X$ is $K$-semimetrizable if and only if there exists a map $\phi:\tau^c\rightarrow\mathcal L(X,\mathcal F(Y))$ satisfying $(G_{
\tau^c})$, $(m_{\tau^c})$, $(b_{\mathcal S_X}(\tau^c))$ and $(b_{\mathcal C_X}(\mathcal C_X))$.
\end{Thm}
\begin{proof}
For each $F\in\tau^c$, define $U(n,F)$ and $\phi_F$ as that in the proof of Theorem \ref{T6}. Then we only need to show $(b_{\mathcal C_X}(\mathcal C_X))$.

Let $K,H\in{\mathcal C}_X$ and $K\cap H=\emptyset$. Then $d(H,K)>0$ and thus there exists $m\in\mathbb{N}$ such that $d(x,y)>\frac{1}{m}$ for each $x\in H$ and
$y\in K$. It follows that $x\notin U(m,K)$ for each $x\in H$. Thus $n_K(x)\leq m$ and so $\phi_K(x)\subset B_m$ for each $x\in H$. This implies that $\phi_K$
is bounded on $H$.

Conversely, for each $x,y\in X$ with $x\neq y$, let $m(x,y)=\min\{n\in\mathbb{N}:y\in\phi^{\#}_x[B_n] \ {\rm and} \ x\in\phi^{\#}_y[B_n]\}$. Define a function
$d:X\times X\rightarrow [0,\infty)$ by letting $d(x,y)=0$ whenever $x=y$ and $d(x,y)=\frac{1}{m(x,y)}$ whenever $x\neq y$. It is easy to verify that $y\in B(x,
\frac{1}{n})$ if and only if $y\notin\phi^{\#}_x[B_n]$ or $x\notin\phi^{\#}_y[B_n]$.

Claim 1. $B(x,r)$ is a neighborhood of $x$ for each $x\in X$ and $r>0$.

{\it Proof.} Let $r>0$ and choose $m\in\mathbb{N}$ such that $\frac{1}{m}<r$. Let $O_x=\phi^{-1}_x[Y\setminus B_m]$. Then $O_x$ is an open neighborhood of $x$
and $O_x\subset B(x,\frac{1}{m})\subset B(x,r)$. This implies that $B(x,r)$ is a neighborhood of $x$.

Claim 2. If $x\notin F\in\tau^c$, then $B(x,r)\cap F=\emptyset$ for some $r>0$.

{\it Proof.} Let $x\notin F\in\tau^c$. By $(b_F(\{x\}))$, $\phi_F(x)\subset B_m$ for some $m\in\mathbb{N}$. By $(b_{\mathcal S_X}(\tau^c))$, there exists $k\geq
m$ such that $y\in\phi^{\#}_x[B_k]$ for each $y\in F$. Also, By $(m_{\tau^c})$, for each $y\in F$, $\phi_y(x)\subset\phi_F(x)\subset B_k$ which implies that $x
\in\phi^{\#}_y[B_k]$. Therefore, $y\notin B(x,\frac{1}{k})$ for each $y\in F$ and thus $B(x,\frac{1}{k})\cap F=\emptyset$.

By Claim 1 and Claim 2, $d$ is a semi-metric on $X$.

Now, let $K,H\in{\mathcal C}_X$ and $K\cap H=\emptyset$. By $(b_{\mathcal C_X}(\mathcal C_X))$, there exists $m\in\mathbb{N}$ such that $K\subset\phi_H^{\#}[B_m]
$ and $H\subset\phi_K^{\#}[B_m]$. Thus for each $x\in K$ and $y\in H$, $x\in\phi_H^{\#}[B_m]\subset\phi_y^{\#}[B_m]$ and $y\in\phi_K^{\#}[B_m]\subset\phi_x^{\#}[
B_m]$. It follows that $m(x,y)\leq m$ and so $d(x,y)=\frac{1}{m(x,y)}\geq\frac{1}{m}$. This implies that $d(K,H)>0$. Therefore, $X$ is a $K$-semimetrizable space.
\end{proof}

\begin{Thm}\label{T8}
For a space $X$, the following are equivalent.

$(a)$ $X$ is $o$-semimetrizable.

$(b)$ There exists a map $\phi:\tau^c\rightarrow\mathcal L(X,\mathcal F(Y))$ satisfying $(G_{\tau^c})$, $(m_{\tau^c})$ and $(S)$.

$(c)$ There exists a map $\phi:\mathcal S_X\rightarrow\mathcal L(X,\mathcal F(Y))$ satisfying $(G_{\mathcal S_X})$, $(b_{\mathcal S_X}(\tau^c))$ and $(S)$.
\end{Thm}
\begin{proof}
(a) $\Rightarrow$ (b) For each $F\in\tau^c$ and $n\in\mathbb{N}$, let $U(n,F)=\cup\{B(x,\frac{1}{n}):x\in F\}$. Then $\bigcap_{n\in\mathbb{N}}U(n,F)=F$. For each
$x\notin F$, let $n_F(x)=\min\{n\in\mathbb{N}:x\notin U(n,F)\}$. For each $F\in\tau^c$, define a map $\phi_F:X\rightarrow\mathcal F(Y)$ by letting $\phi_F(x)=Y$
whenever $x\in F$ and $\phi_F(x)=B_{n_F(x)}$ whenever $x\notin F$. Then $\phi:\tau^c\rightarrow\mathcal L(X,\mathcal F(Y))$ satisfies $(G_{\tau^c})$ and $(m_{
\tau^c})$. For each $x,y\in X$ with $x\neq y$, $y\notin U(n,\{x\})$ if and only if $x\notin U(n,\{y\})$ from which it follows that $n_{\{x\}}(y)=n_{\{y\}}(x)$.
Thus $\phi_x(y)=\phi_y(x)$.

(b) $\Rightarrow$ (c) Assume (b). Suppose that $x\notin F\in\tau^c$. Then $\phi_F(x)\subset B_m$ for some $m\in\mathbb{N}$. Hence, $\phi_x(y)=\phi_y(x)\subset
\phi_F(x)\subset B_m$ for each $y\in F$. This implies that $\phi_x$ is bounded on $F$.

(c) $\Rightarrow$ (a) Assume (c). Define a $g$-function $g$ for $X$ by letting $g(n,x)=\phi_x^{-1}[Y\setminus B_n]$ for each $x\in X$ and $n\in\mathbb{N}$. Let
$x,y\in X$. From $\phi_x(y)=\phi_y(x)$ it follows that for each $n\in\mathbb{N}$, $y\in g(n,x)$ if and only if $x\in g(n,y)$.

By $(b_{\mathcal S_X}(\tau^c))$, $\{g(n,x):n\in\mathbb{N}\}$ is a neighborhood base of $x$. Therefore $X$ is an $o$-semimetrizable space \cite{Go}.
\end{proof}

A space $X$ is called ultrametrizable \cite{d} if there exists a compatible metric $d$ on $X$ such that $d(x,z)\leq\max\{d(x,y),d(y,z)\}$ for each $x,y,z\in X$.

\begin{Thm}\label{T9}
For a space $X$, the following are equivalent.

$(a)$ $X$ is ultrametrizable.

$(b)$ There exists a map $\phi:\tau^c\rightarrow\mathcal L(X,\mathcal F(Y))$ satisfying $(G_{\tau^c})$, $(m_{\tau^c})$, $(S)$ and $(E)$.

$(c)$ There exists a map $\phi:\mathcal S_X\rightarrow\mathcal L(X,\mathcal F(Y))$ satisfying $(G_{\mathcal S_X})$, $(b_{\mathcal S_X}(\tau^c))$, $(S)$ and $(E)$.
\end{Thm}
\begin{proof}
(a) $\Rightarrow$ (b) Let $d$ be a compatible metric on $X$. Define a map $\phi:\tau^c\rightarrow\mathcal L(X,\mathcal F(Y))$ as that in the proof of (a) $
\Rightarrow$ (b) of Theorem \ref{T8}. Then $\phi$ satisfies $(G_{\tau^c})$, $(m_{\tau^c})$ and $(S)$.

Let $x,y,z$ be pairwise distinct points of $X$. Assume that $\phi_x(y)\nsubseteq\phi_x(z)$ and $\phi_y(z)\nsubseteq\phi_x(z)$. Then $n_x(z)<n_x(y)$ and $n_x(z)<n_
y(z)$. Thus $y\in U(n_x(z),\{x\})=B(x,\frac{1}{n_x(z)})$ and $z\in U(n_x(z),\{y\})=B(y,\frac{1}{n_x(z)})$ from which it follows that $d(x,z)\leq\max\{d(x,y),d(y,z)
\}<\frac{1}{n_x(z)}$. Hence $z\in B(x,\frac{1}{n_x(z)})=U(n_x(z),\{x\})$, a contradiction.

(b) $\Rightarrow$ (c) Similar to the proof of (b) $\Rightarrow$ (c) of Theorem \ref{T8}.

(c) $\Rightarrow$ (a) Assume (c). Define a $g$-function for $X$ by letting $g(n,x)=\phi_x^{-1}[Y\setminus B_n]$ for each $x\in X$ and $n\in\mathbb{N}$. By $(S)$, we
have that $y\in g(n,x)$ if and only if $x\in g(n,y)$. By $(b_{\mathcal S_X}(\tau^c))$, $\{g(n,x):n\in\mathbb{N}\}$ is a neighborhood base of $x$. By
$(E)$, we have that if $y\in g(n,x)$, then $g(n,y)\subset g(n,x)$. Therefore, $X$ is ultrametrizable \cite{Go}.
\end{proof}

A space $X$ is called a Nagata-space \cite{Ho} if there exists a $g$-function $g$ for $X$ such that if $g(n,x)\cap g(n,x_n)\neq\emptyset$ for all $n\in\mathbb{N}$,
then $x$ is a cluster point of $\langle x_n\rangle$.

\begin{Thm}\label{T10}
For a space $X$, the following are equivalent.

$(a)$ $X$ is a Nagata-space.

$(b)$ There exists a map $\phi:\tau^c\rightarrow\mathcal L(X,\mathcal F(Y))$ satisfying $(G_{\tau^c})$, $(m_{\tau^c})$, $(b_{\mathcal S_X}(\tau^c))$ and $(\diamond)
$ $\phi_F$ is locally bounded on $X\setminus F$.

$(c)$ There exists a map $\phi:\tau^c\rightarrow\mathcal L(X,\mathcal F(Y))$ satisfying $(G_{\tau^c})$, $(m_{\tau^c})$, $(b_{\mathcal S_X}(\tau^c))$ and $(b_{\tau^
c}(\mathcal C_X))$.
\end{Thm}
\begin{proof}
(a) $\Rightarrow$ (b) Let $g$ be the $g$-function for a Nagata-space. It is easy to verify that $\bigcap_{n\in\mathbb{N}}\overline{g(n,F)}=F$ for each $F\in\tau^c$
and that $\{g(n,x):n\in\mathbb{N}\}$ is a neighborhood base of $x$. For each $F\in\tau^c$ and $x\notin F$, let $n_F(x)=\min\{n\in\mathbb{N}:x\notin g(n,F)\}$. For
each $F\in\tau^c$, define a map $\phi_F$ by letting $\phi_F(x)=Y$ whenever $x\in F$ and $\phi_F(x)=B_{n_F(x)}$ whenever $x\notin F$. Then $\phi:\tau^c\rightarrow
\mathcal L(X,\mathcal F(Y))$ satisfies $(G_{\tau^c})$, $(m_{\tau^c})$, $(b_{\mathcal S_X}(\tau^c))$. Similar to the proof of the necessity of Theorem \ref{T5}, we
can show that $\phi_F$ satisfies $(\diamond)$.

(b) $\Rightarrow$ (c) Assume (b). Let $F\in\tau^c$, $K\in{\mathcal C}_X$ and $F\cap K=\emptyset$. By $(\diamond)$, for each $x\in K$, there exist an open
neighborhood $O_x$ of $x$ and $n_x\in\mathbb{N}$ such that $O_x\subset\phi^{\#}_F[B_{n_x}]$. Since $K\in{\mathcal C}_X$, there exists a finite subset $A$ of $K$
such that $K\subset\cup\{O_x:x\in A\}$. Let $m=\max\{n_x:x\in A\}$. For each $x\in K$, there is $y\in A$ such that $x\in O_y\subset\phi^{\#}_F[B_{n_y}]
\subset\phi^{\#}_F[B_m]$. This implies that $\phi_F$ is bounded on $K$.

(c) $\Rightarrow$ (a) Assume (c). Then by Proposition \ref{P2}, $X$ is first countable. By Theorem \ref{T4}, $X$ is $k$-semi-stratifiable. Thus $X$ is a Nagata
space \cite{Mm}.
\end{proof}

A function $d:X\times X\rightarrow [0,\infty)$ is called a quasi-metric on $X$ if (i) $d(x,y)=0$ if and only if $x=y$; (ii) for all $x,y,z\in X$, $d(x,z)\leq d(
x,y)+d(y,z)$. Let $d^{-1}(x,y)=d(y,x)$ for all $x,y\in X$. Then $d^{-1}$ is also a quasi-metric on $X$. For each $x\in X$ and $r>0$, let $B_d(x,r)=\{y\in X:d(x,
y)<r\}$ and $B_{d^{-1}}(x,r)=\{y\in X:d^{-1}(x,y)<r\}$. Then $\{B_d(x,r):x\in X,r>0\}$ ($\{B_{d^{-1}}(x,r):x\in X,r>0\}$) is a base for a topology $\tau_d$ ($
\tau_{d^{-1}}$) on $X$. If $\tau_d\subset\tau_{d^{-1}}$, then $d$ is called a strong quasi-metric \cite{St}. A space $(X,\tau)$ is called strongly
quasi-metrizable \cite{Ku} if there exists a strong quasi-metric $d$ on $X$ such that $\tau=\tau_d$.

\begin{Thm}\label{T11}
A space $X$ is strongly quasi-metrizable if and only if there exists a map $\phi:\tau^c\rightarrow\mathcal L(X,\mathcal F(Y))$ satisfying $(G_{\tau^c})$,
$(m_{\tau^c})$ and $(b_{\mathcal C_X}(\tau^c))$.
\end{Thm}
\begin{proof}
Let $d$ be a strong quasi-metric on $X$ such that $\tau=\tau_d$. For each $F\in\tau^c$ and $n\in\mathbb{N}$, let $U(n,F)=\cup\{B_d(x,\frac{1}{n}):x\in F\}$. Then
$U(n,F)$ is open and $F\subset U(n,F)$. Let $x\in\bigcap_{n\in\mathbb{N}}U(n,F)$. Then there exists $x_n\in F$ such that $x\in B_d(x_n,\frac{1}{n})$ from which it
follows that $x_n\in B_{d^{-1}}(x,\frac{1}{n})$ for each $n\in\mathbb{N}$. Thus $x_n\rightarrow x$ in $\tau_{d^{-1}}$. Since $\tau_d\subset\tau_{d^{-1}}$, we have
that $x_n\rightarrow x$ in $\tau_d$. Thus $x\in F$ and so $\bigcap_{n\in\mathbb{N}}U(n,F)\subset F$. Therefore, $\bigcap_{n\in\mathbb{N}}U(n,F)=F$.

For each $F\in\tau^c$ and $x\notin F$, let $n_F(x)=\min\{n\in\mathbb{N}:x\notin U(n,F)\}$. For each $F\in\tau^c$, define a map $\phi_F:X\rightarrow\mathcal F(Y)$
by letting $\phi_F(x)=Y$ whenever $x\in F$ and $\phi_F(x)=B_{n_F(x)}$ whenever $x\notin F$. Then $\phi:\tau^c\rightarrow\mathcal L(X,\mathcal F(Y))$ satisfies $(
G_{\tau^c})$ and $(m_{\tau^c})$

Let $F\in\tau^c$, $K\in{\mathcal C}_X$ and $K\cap F=\emptyset$. We show that $F\cap U(m,K)=\emptyset$ for some $m\in\mathbb{N}$. Assume that $F\cap U(n,K)\neq
\emptyset$ for each $n\in\mathbb{N}$. Then there exist $x_n\in F,y_n\in K$ such that $x_n\in B_d(y_n,\frac{1}{n})$. Since $K\in{\mathcal C}_X$, $\langle y_n\rangle$
has a cluster point $x$ in $K$. Then there exists a subsequence $\langle y_{n_k}\rangle$ of $\langle y_n\rangle$ such that $y_{n_k}\in B_d(x,\frac{1}{k})$ for each
$k\in\mathbb{N}$. Since $x_{n_k}\in B_d(y_{n_k},\frac{1}{k})$, we have that $d(x,x_{n_k})\leq d(x,y_{n_k})+d(y_{n_k},x_{n_k})<\frac{2}{k}$ for each $k\in\mathbb{N}$.
It follows that $x_{n_k}\rightarrow x$ and thus $x\in F$, a contradiction to that $K\cap F=\emptyset$. Therefore, $F\cap U(m,K)=\emptyset$ for some $m\in\mathbb{N}$.
For each $x\in F$, $x\notin U(m,K)$ from which it follows that $n_K(x)\leq m$ and thus $\phi_K(x)=B_{n_K(x)}\subset B_m$. This implies that $\phi_K$ is bounded on $F$.

Conversely, by Proposition \ref{P3}, $X$ is semi-stratifiable. By Theorem \ref{T1}, $X$ is a $\gamma$-space. Thus $X$ is strongly quasi-metrizable \cite{Ku}.
\end{proof}

\begin{Thm}\label{T12}
For a space $X$, the following are equivalent.

$(a)$ $X$ is metrizable.

$(b)$ There exists a map $\phi:\tau^c\rightarrow\mathcal L(X,\mathcal F(Y))$ satisfying $(G_{\tau^c})$, $(m_{\tau^c})$, $(S)$ and $(b_{\mathcal C_X}(\tau^c))$.

$(c)$ There exists a map $\psi:\tau^c\rightarrow\mathcal L(X,\mathcal F(Y))$ satisfying $(G_{\tau^c})$, $(m_{\tau^c})$, $(S)$ and $(b_{\tau^c}(\mathcal C_X))$.

$(d)$ There exists a map $\varphi:\tau^c\rightarrow\mathcal L(X,\mathcal F(Y))$ satisfying $(G_{\tau^c})$, $(m_{\tau^c})$, $(b_{\tau^c}(\mathcal C_X))$ and $(b_{
\mathcal C_X}(\tau^c))$.
\end{Thm}
\begin{proof}
(a) $\Rightarrow$ (b) Let $d$ be a compatible metric on $X$. For each $F\in\tau^c$ and $n\in\mathbb{N}$, let $U(n,F)=\cup\{B(x,\frac{1}{n}):x\in F\}$. Then $\bigcap_
{n\in\mathbb{N}}U(n,F)=F$. For each $x\notin F\in\tau^c$, let $n_F(x)=\min\{n\in\mathbb{N}:x\notin U(n,F)\}$. For each $F\in\tau^c$, define a map $\phi_F:X\rightarrow
\mathcal F(Y)$ by letting $\phi_F(x)=Y$ whenever $x\in F$ and $\phi_F(x)=B_{n_F(x)}$ whenever $x\notin F$. Then $\phi:\tau^c\rightarrow\mathcal L(X,\mathcal F(Y))$
satisfies $(G_{\tau^c})$ and $(m_{\tau^c})$ and $(S)$. Similar to the proof of the necessity of Theorem \ref{T11}, we can show that $\phi$ also satisfies $(b_{
\mathcal C_X}(\tau^c))$.

(b) $\Rightarrow$ (c) Assume (b). For each $F\in\tau^c$ and $n\in\mathbb{N}$, let $V(n,F)=\cup\{\phi_x^{-1}[Y\setminus B_n]:x\in F\}$. Then $V(n,F)$ is open and $F
\subset\bigcap_{n\in\mathbb{N}}V(n,F)$.

{\bf Claim}. If $F\in\tau^c$, $K\in{\mathcal C}_X$ and $F\cap K=\emptyset$, then $K\cap V(m,F)=\emptyset$ for some $m\in\mathbb{N}$.

{\it Proof}. Let $F\in\tau^c$, $K\in{\mathcal C}_X$ and $F\cap K=\emptyset$. By $(b_{\mathcal C_X}(\tau^c))$, $F\subset\phi^{\#}_K[B_m]$ for some $m\in\mathbb{N}$.
For each $x\in K$ and $y\in F$, $\phi_y(x)=\phi_x(y)\subset\phi_K(y)\subset B_m$ which implies that $x\notin\phi_y^{-1}[Y\setminus B_m]$. Thus $x\notin V(m,F)$.
It follows that $K\cap V(m,F)=\emptyset$.

The above claim also implies that $\bigcap_{n\in\mathbb{N}}V(n,F)\subset F$. Thus $\bigcap_{n\in\mathbb{N}}V(n,F)=F$.

For each $x\notin F\in\tau^c$, let $m_F(x)=\min\{n\in\mathbb{N}:x\notin V(n,F)\}$. For each $F\in\tau^c$, define a map $\psi_F:X\rightarrow\mathcal F(Y)$ by letting
$\psi_F(x)=Y$ whenever $x\in F$ and $\psi_F(x)=B_{m_F(x)}$ whenever $x\notin F$. Then $\psi:\tau^c\rightarrow\mathcal L(X,\mathcal F(Y))$ satisfies $(G_{\tau^c})$
and $(m_{\tau^c})$.

For each $x,y\in X$ with $x\neq y$, since $\phi_x(y)=\phi_y(x)$, we have that $y\in V(n,\{x\})$ if and only if $x\in V(n,\{y\})$ for each $n\in\mathbb{N}$. It follows
that $m_{\{x\}}(y)=m_{\{y\}}(x)$ and thus $\psi_x(y)=\psi_y(x)$.

From the above claim it follows that $\psi$ satisfies $(b_{\tau^c}(\mathcal C_X))$.

(c) $\Rightarrow$ (d) Assume (c). For each $F\in\tau^c$ and $n\in\mathbb{N}$, let $W(n,F)=\cup\{\psi_x^{-1}[Y\setminus B_n]:x\in F\}$. Then $W(n,F)$ is open and $F
\subset\bigcap_{n\in\mathbb{N}}W(n,F)$. By $(b_{\tau^c}(\mathcal C_X))$, if $x\notin F$, then $\psi_F(x)\subset B_m$ for some $m\in\mathbb{N}$. Then for each $y\in
F$, $\psi_y(x)\subset\psi_F(x)\subset B_m$ from which it follows that $x\notin\psi_y^{-1}[Y\setminus B_m]$. Thus $x\notin W(m,F)$. This implies that $\bigcap_{n\in
\mathbb{N}}W(n,F)\subset F$ and hence $F=\bigcap_{n\in\mathbb{N}}W(n,F)$.

For each $x\notin F\in\tau^c$, let $i_F(x)=\min\{n\in\mathbb{N}:x\notin W(n,F)\}$. For each $F\in\tau^c$, define $\varphi_F:X\rightarrow\mathcal F(Y)$ by letting
$\varphi_F(x)=Y$ whenever $x\in F$ and $\varphi_F(x)=B_{i_F(x)}$ whenever $x\notin F$. Then $\varphi:\tau^c\rightarrow\mathcal L(X,\mathcal F(Y))$ satisfies $(G_{
\tau^c})$ and $(m_{\tau^c})$.

Let $F\in\tau^c$, $K\in{\mathcal C}_X$ and $F\cap K=\emptyset$. By $(b_{\tau^c}(\mathcal C_X))$, $K\subset\psi^{\#}_F[B_m]$ for some $m\in\mathbb{N}$. For each $x
\in K$ and $y\in F$, $\psi_y(x)\subset\psi_F(x)\subset B_m$ which implies that $x\notin\psi_y^{-1}[Y\setminus B_m]$. Thus $x\notin W(m,F)$ and so $i_F(x)\leq m$.
Consequently, $\varphi_F(x)=B_{i_F(x)}\subset B_m$ for each $x\in K$. This implies that $\varphi_F$ is bounded on $K$. Also, since $\psi_x(y)=\psi_y(x)$, we have
that $y\notin\psi_x^{-1}[Y\setminus B_m]$ for each $x\in K$ and $y\in F$ from which it follows that $y\notin W(m,K)$ and so $i_K(y)\leq m$. Consequently, $\varphi_
K(y)=B_{i_K(y)}\subset B_m$ for each $y\in F$. This implies that $\varphi_K$ is bounded on $F$.

(d) $\Rightarrow$ (a) By Theorem \ref{T1}, $X$ is a $\gamma$-space. By Theorem \ref{T4}, $X$ is a $k$-semi-stratifiable space. Therefore, $X$ is metrizable \cite{Yo}.
\end{proof}

\section{An application}
As an application, in this section, we use the results obtained in the previous sections to give characterizations of the corresponding spaces with
real-valued functions.

Let $\mathbb R$ be the real number set. $\overline{\mathbb R}=\mathbb R\cup\{-\infty,+\infty\}$ is called the generalized real number set. The order $<_{\overline
{\mathbb R}}$ on $\overline{\mathbb R}$ is defined as follows: $<_{\overline{\mathbb R}}=<_{\mathbb R}\cup\{\langle-\infty,r\rangle,\langle r,+\infty\rangle,r\in
\mathbb R\}\cup\{\langle-\infty,+\infty\rangle\}$. When viewed as a topological space, the topology on $\overline{\mathbb R}$ takes the collection $\{[-\infty,a),
(a,b),(a,+\infty]:a,b\in\mathbb R\}$ as a base.

A generalized real-valued function $f:X\rightarrow\overline{\mathbb R}$ is called lower (upper) semi-continuous \cite{T}, if for each $r\in\mathbb R$, the set
$\{x\in X:f(x)>r\}$ ($\{x\in X:f(x)<r\}$) is open.

The following lemmas provide methods of turning a lower semi-continuous set-valued map $\phi:X\rightarrow\mathcal F(\overline{\mathbb R})$ into a lower
semi-continuous function $f:X\rightarrow\overline{\mathbb R}$ and of the converse.

\begin{Lem}\label{L3}
If a set-valued map $\phi:X\rightarrow\mathcal F(\overline{\mathbb R})$ is lower semi-continuous, then the generalized real-valued function $f:X\rightarrow
\overline{\mathbb R}$ defined by letting $f(x)=\sup(\phi(x)\setminus\{+\infty\})$ for each $x\in X$, is lower semi-continuous.
\end{Lem}
\begin{proof}
Let $r\in\mathbb R$ and $f(x)>r$. Then $(\phi(x)\setminus\{+\infty\})\cap(r,+\infty]\neq\emptyset$ and thus $\phi(x)\cap(r,+\infty)\neq\emptyset$. Since $\phi$
is lower semi-continuous, $U=\phi^{-1}[(r,+\infty)]$ is an open neighborhood of $x$. For each $y\in U$, $\phi(y)\cap(r,+\infty)\neq\emptyset$ from which it
follows that $f(y)=\sup(\phi(y)\setminus\{+\infty\})>r$. Therefore, $f$ is lower semi-continuous.
\end{proof}

\begin{Lem}\label{L4}
A generalized real-valued function $f:X\rightarrow\overline{\mathbb R}$ is lower semi-continuous if and only if the set-valued map $\varphi:X\rightarrow\mathcal
F(\overline{\mathbb R})$ defined by letting $\varphi(x)=[-\infty,f(x)]$ for each $x\in X$, is lower semi-continuous.
\end{Lem}
\begin{proof}
Let $V$ be an open subset of $\overline{\mathbb R}$ and $\varphi(x)\cap V\neq\emptyset$. Choose $a\in\varphi(x)\cap V$. Then $a\leq f(x)$.

Case 1. $a=+\infty$. Then $f(x)=+\infty$. Since $+\infty=a\in V$, there exists $r\in\mathbb{R}$ such that $(r,+\infty]\subset V$. Let $U=\{x\in X:f(x)>r\}$. Since
$f$ is lower semi-continuous and $f(x)=+\infty$, $U$ is an open neighborhood of $x$. For each $y\in U$, $f(y)>r$ and thus $[-\infty,f(y)]\cap(r,+\infty]\neq
\emptyset$ from which it follows that $\varphi(y)\cap V\neq\emptyset$.

Case 2. $a=-\infty$. Then $-\infty\in V$. For each $y\in X$, $-\infty\in\varphi(y)\cap V$.

Case 3. $a\in\mathbb R$. Then there exist $s,r\in\mathbb{R}$ with $s<r$ such that $a\in(s,r)\subset V$. Let $U=\{x\in X:f(x)>s\}$. Since $f$ is lower semi-continuous
and $f(x)>s$, $U$ is an open neighborhood of $x$. For each $y\in U$, $f(y)>s$.

Case 3.1. $f(y)\geq r$. Then $(s,r)\subset[-\infty,f(y)]$ and so $(s,r)\subset\varphi(y)\cap V$.

Case 3.2. $f(y)<r$. Then $f(y)\in(s,r)\subset V$ and so $f(y)\in\varphi(y)\cap V$.

The above argument shows that $\varphi$ is lower semi-continuous.

Conversely, let $r\in\mathbb{R}$ and $f(x)>r$. Then $\varphi(x)\cap(r,\infty]\neq\emptyset$. Since $\varphi$ is lower semi-continuous, $U=\varphi^{-1}[(r,\infty
]]$ is an open neighborhood of $x$. For each $y\in U$, $[-\infty,f(y)]\cap(r,\infty]\neq\emptyset$ which implies that $f(y)>r$. Therefore, $f$ is lower
semi-continuous.
\end{proof}

For the space $\overline{\mathbb R}$, if we set $B_n=[-\infty,n]\cup\{+\infty\}$ for each $n\in\mathbb{N}$, then $\{B_n\}_{n\in\mathbb{N}}$ is a strictly
increasing closed cover of $\overline{\mathbb R}$. Thus as a special case, the results in the previous sections hold for $Y=\overline{\mathbb R}$. Then with the
above lemmas, we can translate the set-valued maps with values in $\mathcal F(\overline{\mathbb R})$ of the corresponding results into generalized real-valued
functions. We take stratifiable spaces as an example. The other results can be stated and proved analogously.

\begin{Thm}\label{T13}
A space $X$ is stratifiable if and only if for each $F\in\tau^c$, there exists a lower semi-continuous function $f_F:X\rightarrow\overline{\mathbb R}$ such that
$(1)$ $F=f_F^{-1}(+\infty)$; $(2)$ if $F_1\subset F_2$, then $f_{F_1}\leq f_{F_2}$; $(3)$ if $x\notin F\in\tau^c$, then there exists an open
neighborhood $V$ of $x$ and $m\in\mathbb{N}$ such that $f_F(y)\leq m$ for each $y\in V$.
\end{Thm}
\begin{proof}
For each $n\in\mathbb{N}$, let $B_n=[-\infty,n]\cup\{+\infty\}$. Then $\{B_n\}_{n\in\mathbb{N}}$ is a strictly increasing closed cover of $\overline{\mathbb R}$.
By letting $Y=\overline{\mathbb R}$ in Theorem \ref{T5}, we have the following.

{\bf Claim}. A space $X$ is stratifiable if and only if for each $F\in\tau^c$, there exists a lower semi-continuous map $\phi_F:X\rightarrow\mathcal F(\overline
{\mathbb R})$ such that $(1')$ $F=\bigcap_{n\in\mathbb{N}}\phi_F^{-1}[(n,+\infty)]$; $(2')$ if $F_1\subset F_2$, then $\phi_{F_1}\subset\phi_{F_2}$; $(3')$ if $x
\notin F\in\tau^c$, then there exists an open neighborhood $V$ of $x$ and $m\in\mathbb{N}$ such that $\phi_F(y)\subset B_m$ for each $y\in V$.

Let $X$ be a stratifiable space and $\phi_F$ be the map in the above claim for each $F\in\tau^c$. For each $x\in X$, let $f_F(x)=\sup(\phi_F(x)\setminus\{+\infty
\})$. By Lemma \ref{L3}, $f_F:X\rightarrow\overline{\mathbb R}$ is lower semi-continuous. It is clear that $(2)$ holds. By $(1')$, if $x\in F$, then $\phi_F(x)
\cap(n,+\infty)\neq\emptyset$ for each $n\in\mathbb{N}$. Choose $x_n\in\phi_F(x)\cap(n,+\infty)$ for each $n\in\mathbb{N}$. Then $\sup\{x_n:n\in\mathbb{N}\}=+
\infty$ and $\{x_n:n\in\mathbb{N}\}\subset\phi_F(x)\setminus\{+\infty\}$ from which it follows that $f_F(x)=+\infty$. On the other hand, if $f_F(x)=+\infty$,
then $(\phi_F(x)\setminus\{+\infty\})\cap(n,+\infty)\neq\emptyset$ and so $\phi_F(x)\cap(n,+\infty)\neq\emptyset$ for each $n\in\mathbb{N}$. This implies
that $x\in\bigcap_{n\in\mathbb{N}}\phi_F^{-1}[(n,+\infty)]$. By $(1')$, $x\in F$.

Now, suppose that $x\notin F\in\tau^c$. Then by $(3')$, there exists an open neighborhood $V$ of $x$ and $m\in\mathbb{N}$ such that $\phi_F(y)\subset B
_m$ for each $y\in V$. Thus $\phi_F(y)\setminus\{+\infty\}\subset[-\infty,m]$ and so $f_F(y)\leq m$.

Conversely, for each $F\in\tau^c$, define a map $\phi_F:X\rightarrow\mathcal F(\overline{\mathbb R})$ by letting $\phi_F(x)=[-\infty,f_F(x)]$ for each $x\in X$.
By Lemma \ref{L4}, $\phi_F$ is lower semi-continuous. It is clear that if $F_1\subset F_2$, then $\phi_{F_1}\subset\phi_{F_2}$. If $x\in F$, then $\phi_F(x)=
\overline{\mathbb R}$ and thus $x\in\bigcap_{n\in\mathbb{N}}\phi_F^{-1}[(n,+\infty)]$. If $x\in\bigcap_{n\in\mathbb{N}}\phi_F^{-1}[(n,+\infty)]$, then $\phi_F(x)
\cap(n,+\infty)\neq\emptyset$ for each $n\in\mathbb{N}$. It follows that $f_F(x)=+\infty$ and thus $x\in F$.

Suppose that $x\notin F\in\tau^c$. Then there exists an open neighborhood $V$ of $x$ and $m\in\mathbb{N}$ such that $f_F(y)\leq m$ for each
$y\in V$. Thus $\phi_F(y)\subset B_m$. By the above claim, $X$ is a stratifiable space.
\end{proof}

\end{document}